\documentclass[11pt]{amsart}

\usepackage{amsfonts, amssymb, amsmath, mathrsfs}
\usepackage{tikz-cd}
\usepackage{mathtools}
\usepackage[hidelinks]{hyperref}
\usepackage{url}
\usepackage[left=25mm,right=25mm,top=38mm,bottom=34mm]{geometry}
\newcommand{\spa}{\medskip}
\usepackage[utf8]{inputenc}

\usepackage{bbm}

\newcommand{\stack}[1]{\cite[\href{https://stacks.math.columbia.edu/tag/#1}{Tag #1}]{Stacks}}

\newcommand{\ag}{{\mbox{\larger[-4]$\dag$}}}

\newcommand{\Q}{\mathbb{Q}}
\newcommand{\F}{\mathbb{F}}

\newcommand{\Z}{\mathbb{Z}}
\newcommand{\N}{\mathbb{N}}

\newcommand{\bt}{$p$-divisible\ }
\newcommand{\BT}{\mathbf{BT}}
\newcommand{\DM}{\mathbf{DM}}

\newcommand{\Foi}{\Fisoc^\mathrm{\ag}}

\newcommand{\iso}{\xrightarrow{\sim}}

\newcommand{\Zp}{\mathbb{Z}_p}

\newcommand{\Qp}{{\mathbb{Q}_p}}

\newcommand{\QQ}{\mathbb{Q}}

\newcommand{\tors}{\mathrm{tors}}

\newcommand{\Fp}{{\F_p}}

\newcommand{\calE}{\mathcal{E}}

\newcommand{\calF}{\mathcal{F}}

\newcommand{\calG}{\mathcal{G}}

\newcommand{\calH}{\mathcal{H}}

\newcommand{\calO}{\mathcal{O}}

\newcommand{\calM}{\mathcal{M}}

\newcommand{\Spec}{\mathrm{Spec}}

\newcommand{\Gal}{\mathrm{Gal}}

\newcommand{\End}{\mathrm{End}}

\newcommand{\GL}{\mathrm{GL}}

\newcommand{\et}{{\mathrm{\acute{e}t}}}
\newcommand{\gp}{\mathrm{gp}}

\newcommand{\Fisoc}{\mathbf{F\textrm{-}Isoc}}

\begin{document}

	\newtheorem{theo}[subsection]{Theorem}
	\newtheorem*{theo*}{Theorem}
	\newtheorem{ques}[subsection]{Question}
	\newtheorem*{ques*}{Question}
	\newtheorem{conj}[subsection]{Conjecture}
	\newtheorem{prop}[subsection]{Proposition}
	\newtheorem{lemm}[subsection]{Lemma}
	\newtheorem*{lemm*}{Lemma}
	\newtheorem{coro}[subsection]{Corollary}
	\newtheorem*{coro*}{Corollary}

	\theoremstyle{definition}
	\newtheorem{defi}[subsection]{Definition}
	\newtheorem*{defi*}{Definition}
	\newtheorem{hypo}[subsection]{Hypothesis}
	\newtheorem{rema}[subsection]{Remark}
	\newtheorem{exam}[subsection]{Example}
	\newtheorem{nota}[subsection]{Notation}
	\newtheorem{cons}[subsection]{Construction}
	
	\numberwithin{equation}{section}
	
	\title{Logarithmic Dieudonné theory and overconvergent extensions}
	\date{\today}
	\makeatletter
\@namedef{subjclassname@2020}{%
	\textup{2020} Mathematics Subject Classification}
\makeatother



	\author{Marco D'Addezio}
\address{Institut de Recherche Math\'ematique Avanc\'ee (IRMA), Universit\'e de Strasbourg, 7 rue Ren\'e-Descartes, 67000 Strasbourg, France}
\email{daddezio@unistra.fr}

	\begin{abstract}
	In the proof of Crew's parabolicity conjecture, we established a key property concerning the slopes of $\dagger$-hulls of $F$-isocrystals, extending a result of Tsuzuki. This article presents an alternative proof of this theorem for a specific class of $F$-isocrystals. The central ingredient is a local extension property for étale $p$-divisible subgroups. To relate $p$-divisible groups and overconvergent $F$-isocrystals, we employ logarithmic Dieudonné theory, as introduced by Kato and further developed by Inoue. Over curves, this leads to an equivalence between the category of potentially semi-stable $p$-divisible groups and overconvergent $F$-isocrystals with slopes in the interval $[0,1]$.

	\end{abstract}
	
	\maketitle
	\tableofcontents

	\section{Introduction}
Let $k$ be a perfect field of positive characteristic $p$.  For a smooth variety $X$ over $k$, the category $\mathbf{F\textrm{-}Isoc}(X)$ of (convergent) $F$-isocrystals and the category $\mathbf{F\textrm{-}Isoc}^\ag(X)$ of overconvergent $F$-isocrystals are the categories of coefficients of convergent and overconvergent cohomology, as introduced by Berthelot. By \cite{KedFull}, the natural functor $\iota\colon\mathbf{F\textrm{-}Isoc}^\ag(X)\hookrightarrow\mathbf{F\textrm{-}Isoc}(X)$ is fully faithful. An object in the essential image of $\iota$ is called \textit{$\dag$-extendable}.  A basic question is to understand which convergent $F$-isocrystals are $\dag$-extendable.  This article studies the overconvergence of quotients of a $\dag$-extendable $F$-isocrystal under a certain slope condition.  Our main result is the following theorem.

\begin{theo}[{Theorem \ref{a-UR01:t}}]\label{intro-a-UR01:t}
	Let $\calE$ be a $\dag$-extendable $F$-isocrystal over $X$ whose $p$-adic slopes lie in an interval $[s,s+1]$ for some $s\in\QQ$.  Then every quotient of $(\calE,\Phi_\calE)$ in $\mathbf{F\textrm{-}Isoc}(X)$ which is isoclinic of slope $s$ is again $\dag$-extendable.
\end{theo}

We use the following local extension property for étale subgroups of $p$-divisible groups proved by Tate.

\begin{prop}[{\cite[Prop. 12]{Tat67}}]\label{good-red-et-intro:p}Let $R$ be a discrete valuation ring. For a $p$-divisible group $\calG_R$ over $R$, every étale \bt subgroup $\calH_\eta\subseteq\calG_\eta$ of the generic fibre has good reduction.
\end{prop}

To connect $p$-divisible groups with overconvergent $F$-isocrystals we use \textit{log Dieudonné theory}, introduced by Kato \cite{Kat23} and extended by Inoue \cite{Ino25}. Their theory works particularly well for \textit{Frobenius‑smooth log schemes} (Definition \ref{d:FrobeniusLogSmooth}) and provides an equivalence (Theorem \ref{KatoEquivalence:t}). Thanks to this, we prove that, over a curve, there is an equivalence between \textit{potentially semi‑stable $p$-divisible groups}, defined up-to-isogeny, and overconvergent $F$-isocrystals with slopes in $[0,1]$.

\begin{theo}[{Theorem \ref{DM:t}}]\label{intro-DM:t}
	Let $X$ be a smooth curve over $k$. There exists an equivalence
	$$\mathbb{D}^\ag:\BT^{\mathrm{pst}}(X)_\Qp\iso \Foi(X)^{[0,1]}$$ of $\Qp$-linear abelian categories, where $\BT^{\mathrm{pst}}(X)$ denotes the category of potentially semi-stable $p$-divisible groups over $X$.
\end{theo}

Theorem \ref{intro-DM:t} generalises earlier results in \cite{Tri08, Pal22}.  Combining Theorem \ref{intro-DM:t} with Proposition \ref{good-red-et-intro:p}, and using Tsuzuki's characterisation of unit-root overconvergent $F$-isocrystals, \cite{TsuFLM}, we obtain Theorem \ref{intro-a-UR01:t}. A technical ingredient in the proof of Theorem \ref{intro-DM:t} is that, in small dimension, the isogeny category of log Dieudonné modules is exactly the category of coherent log $F$-isocrystals with slopes in $[0,1]$ (Proposition \ref{DMIsoc:p}). 

\spa

For log $p$-divisible groups Theorem \ref{intro-a-UR01:t} acquires the following form.

\begin{theo}[Theorem \ref{et-sub-psemi:t}]\label{et-sub-psemi-intro:t}
	Let $(Y,D)$ be a smooth compactification of $X$, where $D$ is a normal crossing divisor, and let $\calG$ be a log $p$-divisible group over $(Y,D)$.  Then every étale $p$-divisible subgroup $\calH\subseteq\calG_X$ extends to a $p$-divisible group over $Y$.
\end{theo}

A weaker form of Theorem \ref{intro-a-UR01:t} was initially announced in \cite{DA19}. At the time, a full proof was not possible because we could not prove Theorem \ref{intro-DM:t}. Nevertheless, the extension property of étale $p$-divisible subgroups provided fundamental motivation, supporting the hope that \textit{Crew's parabolicity conjecture} could be solved \cite{DA23}. This conjecture roughly states that the monodromy group of a $\dag$-extendable $F$-isocrystal is ``maximal'' inside its overconvergent monodromy group. It plays an important role in the study of special fibres of Shimura varieties, where it implies that the universal $p$-divisible groups have ``big monodromy'' (see \cite{Igusa,DvH22, J23, J25, LS25}). Theorem \ref{intro-a-UR01:t} is a special case of the following result, which is the main $p$-adic input in the proof of Crew's conjecture.

\begin{theo}[{\cite[Prop. 3.1]{Tsu19}, \cite[Thm. 4.1.3]{DA23}}]\label{Parabolicity:t}
	Let $\calE$ be a $\dag$-extendable $F$-isocrystal over $X$ with constant Newton polygon and let $\calF$ be a sub-$F$-isocrystal of $\calE$. If $\tilde{\calF}\subseteq \calE$ is the smallest $\dag$-extendable sub-$F$-isocrystal containing $\calF$, then the minimal slope steps of $\calF$ and $\tilde{\calF}$ coincide\footnote{In \cite{DA23}, the $F$-isocrystal $\tilde{\calF}$ is called  the \textit{$\dagger$-hull} of $\calF\subseteq \calE$.}.
\end{theo}

 After taking duals, Theorem \ref{intro-a-UR01:t} corresponds to the special case where $(\calE,\Phi_\calE)$ has slopes in $[-s-1,-s]$ and $(\calF,\Phi_\calF)$ has slope $-s$. Indeed, under these conditions, Theorem \ref{Parabolicity:t} forces $\calF=\tilde{\calF}$, which is precisely Theorem \ref{intro-a-UR01:t}. Besides, it is worth noting that while both \cite[Thm. 1.1.1]{MaximalTori} and Theorem \ref{intro-a-UR01:t} are implied by Theorem \ref{Parabolicity:t}, they do not imply each other. Theorem \ref{intro-a-UR01:t} is enough to deduce the following consequence of Theorem \ref{Parabolicity:t} that we proved in \cite{DA23}.
 
  \begin{theo}[{\cite[Thm. 5.2.2]{DA23}}]\label{AbelianIntro:t}
 	Let $\Omega$ be a finitely generated field extension of $\Fp$ and let $A$ be an abelian variety over $\Omega$. The group $A(\Omega^{\mathrm{sep}})[p^\infty]$ is finite in the following two cases. 
 	\begin{itemize}
 		\item[(i)] If $\End(A)\otimes_\Z \Qp$ is a division algebra.
 		\item[(ii)] If $\End(A)\otimes_\Z \Q$ has no factor of Albert-type IV.
 	\end{itemize}
 \end{theo}
Thus, using Proposition \ref{good-red-et-intro:p}, we obtain a more direct proof of Theorem \ref{AbelianIntro:t}. In general, we hope that the proof of Theorem \ref{intro-a-UR01:t} provided in this article can help the reader to better understand the content of Theorem \ref{Parabolicity:t} and suggest further applications.

\subsection*{Overview of the paper}
In §\ref{sec:logflat} we review the notions of log flatness and homological log flatness, and we introduce the concept of Frobenius‑smooth log schemes (Definition \ref{d:FrobeniusLogSmooth}).  Section \ref{sec:logDieudonne} is devoted to log Dieudonné theory; we prove the equivalence for Frobenius‑smooth log schemes (Theorem \ref{KatoEquivalence:t}).  In §\ref{sec:overconvequiv} we establish the overconvergent equivalence between potentially semi‑stable $p$-divisible groups and overconvergent $F$-isocrystals with slopes in $[0,1]$ (Theorem \ref{DM:t}).  Finally, in §\ref{sec:main} we prove the local extension property for étale subgroups (Corollary \ref{local-et-sub:c}) and deduce our main result (Theorem \ref{a-UR01:t}).

\subsection*{Acknowledgements}
We are grateful to Damian Rössler for bringing Proposition \ref{good-red-et-intro:p} to our attention. We thank Emiliano Ambrosi for his collaboration in the writing of \cite{MaximalTori}, which motivated this work. We also thank Kentaro Inoue for answering our questions about his article. Finally, we thank Kentaro Inoue and Fabien Trihan for their helpful comments on a first draft of this paper; their suggestions improved the clarity and correctness of the article.

\spa
The author was supported by the Max Planck Institute for Mathematics, the Deutsche Forschungsgemeinschaft (project IDs: 390685689 and 461915680), the Marie Sklodowska-Curie Actions (project ID: 101068237), and the French National Research Agency (project ID: ANR‑25‑CE40‑7869‑01).

\section{Logarithmic flatness}\label{sec:logflat}
In this article, we use the notion of log schemes as introduced in \cite{Kat89}. This section reviews the basic notions of log flatness required for our study of \textit{log $p$-divisible groups}. We recall the definitions of \textit{log flatness} and \textit{homological log flatness} as introduced in \cite{Kat21} and \cite[\S2.5]{KY25}, with particular attention to morphisms of \textit{Kummer type}. In Lemma \ref{faithful-flatness-lemma} we prove a certain flatness properties for ring homomorphisms induced by Kummer type morphisms. Subsequently, we introduce and study \textit{Frobenius-smooth log schemes} (Definition \ref{d:FrobeniusLogSmooth}), a class of log schemes for which the absolute Frobenius is log quasi-syntomic and log flat of Kummer type. These schemes will serve as the natural setting for \textit{log Dieudonné theory}.

\begin{defi}\label{log-flat:d}
	A morphism $f: (Y, \calM_Y) \to (X, \calM_X)$ of fine log schemes is called \textit{log flat} if fppf locally on $X$ and on $Y$, there exists a chart $\calM_Y \leftarrow P \to Q \to \calM_X$ of $f$ satisfying the following two conditions.
	
	\begin{itemize}
\item[(i)] The induced map $P^{\gp}\to Q^{\gp}$ is injective\footnote{This condition is equivalent to the condition that $\Spec(\calO_Y[Q^\gp])\to \Spec(\calO_Y[P^\gp])$ is flat in the classical sense. Indeed, if $P^{\gp}\to Q^{\gp}$ is injective, then $\calO_Y[Q^\gp]$ is a free $\calO_Y[P^\gp]$-module, with basis given by a set of representatives of $Q^\gp/P^\gp$.}.
\item[(ii)] The induced morphism of schemes $Y \to X\times_{\Spec(\Z[P])}\Spec(\Z[Q])$ is flat in the classical sense.
	\end{itemize}

We also say that $f$ is \textit{homologically log flat} if the following two conditions are satisfied.
\begin{itemize}
\item[{(i')}]  $f\colon Y\to X$ is classically flat.
\item[{(ii')}]  Fppf locally on $X$ and $Y$, there exists a chart $\calM_Y \leftarrow P \to Q \to \calM_X$ such that $P\to Q$ is a flat morphism of monoids, as in \cite[Def. 4.8]{Bha12}. 
\end{itemize} Condition (ii') is satisfied, for example, when the map $\Z[P]\to \Z[Q]$ is flat [\textit{ibid.}, Prop. 4.9].

 \spa
 
 Finally, we say that $f$ is \textit{log quasi-syntomic} if $f$ is homologically log flat and the relative log cotangent complex $\mathbb{L}_{(Y,\calM_Y)/(X,\calM_X)}$ is concentrated in degrees $-1$ and $0$.
 
 \spa
 
	A morphism of fine saturated monoids $P \to Q$ is called of \textit{Kummer type} if it is injective and for every element $q \in Q$ there exists an integer $n > 0$ such that $nq \in P$. We extend the notion to maps of fine saturated log schemes.
\end{defi}

\begin{defi}
The \textit{Kummer log flat topology} or \textit{kfl topology} is the topology on fine saturated log schemes whose coverings are topologically surjective log flat maps of Kummer type which are locally of finite presentation\footnote{This topology generalises the fppf topology rather than the fpqc topology.} \cite[Def. 2.3]{Kat21}. We say that a locally free $\calO$-module of finite rank with respect to this topology is a \textit{kfl vector bundle} (see \S \ref{s:kflVectorBundles}). The main results on \textit{kfl descent} are proved in \cite{INT13} and \cite{Kat21}.
\end{defi}

\begin{exam}
The morphism $(\Z[x,y],x^\N \oplus y^\N)^a\to (\Z[x,y/x],x^\N\oplus  (y/x)^\N)^a$ is log flat but it is not of Kummer type and it is not homologically flat. If $M\subseteq \N^2$ is the monoid generated by $(2,0), (1,1),$ and $(0,2)$, then $(\Z[M],M)^a\to (\Z[\N^{2}],\N^{ 2})^a$ is log flat of Kummer type, but it is not homologically flat. 
\end{exam}

\begin{defi}\label{LocallyFree-log:d}
	A fine log structure $\calM_X$ of a scheme $X$ is \textit{locally free} if étale locally on $X$ there exists a chart of the form $\N^r\to \calM_X$ for some $r\geq 0$.
\end{defi}
\begin{lemm}\label{l:LocFree}	A fine log structure $\calM_X$ is {locally free} if and only if for every geometric point $\bar{x}$, the stalk $(\calM_{X}/\calO^*_X)_{\bar{x}}$ is free.
\end{lemm}
\begin{proof}
To prove the lemma it is enough to show that for a pre-log ring $\alpha\colon \N^r\to R$, the associated log ring $\alpha\colon M\to R$ has the property that $M/R^*$ is free. By definition, $M/R^*$ is $\N^r/\sim$, where the equivalence relation is such that $x\sim 0$ if and only if $\alpha(x)\in R^*$. If $x\sim 0$ and $x=\sum_{i=1}^n x_i$, then each $x_i\sim 0$ because $\alpha(x_i)$ is forced to be a unit. This implies that there exists a subset $I\subseteq \{1,\dots,r\}$ such that $(\N^r/\sim)= \N^r/\N^I$. In particular, $M/R^*$ is free. 
\end{proof}

\begin{lemm}\label{l:Thick}Let $(X,\calM_X)\to (T,\calM_T)$ be a strict thickening of fine log schemes. Then $(X,\calM_X)$ is locally free if and only if $ (T,\calM_T)$ is so.
\end{lemm}
\begin{proof}
After going étale locally, we are reduced to consider a strict morphism of log rings $(B,M_B)\to (A,M_A)$ with $B\twoheadrightarrow A$ surjective with nilpotent kernel. Since $B^*=B\times_A A^*$ as sets, strictness implies that $M_A=M_B\sqcup_{B^*} A^*$ as monoids. In particular, we have that $M_B/B^*=M_A/A^*$. We conclude thanks to Lemma \ref{l:LocFree}.
\end{proof}
\begin{lemm}\label{faithful-flatness-lemma}
	Let $f\colon P \hookrightarrow Q$ be a morphism of Kummer type of fs monoids  and assume that $P$ is free. Then the induced ring homomorphism $\Z[P] \to \Z[Q]$ is faithfully flat of finite type. 
\end{lemm}

\begin{proof}

 Although the statement follows from \cite[Prop. 4.1]{Kat89}, we provide a detailed argument for the reader's convenience. Let $Q_{\tors}$ denote the torsion subgroup of $Q$. The torsion free quotient $\bar{Q} \coloneqq Q/Q_{\tors}$ is an fs monoid sitting inside $P^{\gp}$. The induced map $P \to \bar{Q}$ is Kummer as well. Let $n$ be a positive integer such that $n\bar{Q} \subseteq P$. Choose an identification $\N^r=P$, which induces the inclusion $\bar{Q}\hookrightarrow \tfrac{1}{n}\N^r$. Since $\bar{Q}$ is saturated and $\N^r\to \bar{Q}$ is of Kummer type, we deduce that $\tfrac{1}{n}\N^r\cap \bar{Q}^\gp=\bar{Q}$.

\spa

We choose a set of representatives $F \subseteq \bar{Q}^\gp$ for the finite quotient $\bar{Q}^{\gp}/\Z^r$ contained in the fundamental domain $[0,1)^r \subseteq \frac{1}{n}\Z^r$. By construction, we have $(f+\Z^r)\cap \tfrac{1}{n}\N^r=f+\N^r$ for every $f\in F$. Thus every element of $\bar{Q}$ can be written uniquely as $f+ x$ with $f \in F$ and $x \in \N^r$. Lifting $F$ to a subset $\tilde{F} \subseteq Q$, we obtain that every element of $Q$ can be written uniquely as $\tilde{f} +x + t$ with $\tilde{f} \in \tilde{F}$, $x \in \N^r$, and $t \in Q_{\tors}$. This shows that $\Z[Q]$ is a free $\Z[P]$-module of rank $|Q^{\mathrm{gp}}/P^{\mathrm{gp}}|$, hence $\Z[P] \to \Z[Q]$ is faithfully flat of finite type.
\end{proof}

\begin{coro}\label{c:KummerClassicallyFlat}
	Let $f: (Y, \calM_Y) \to (X, \calM_X)$ be a log flat morphism of Kummer type of fine saturated log schemes. If $\calM_X$ is locally free, then $f$ is homologically log flat. 
\end{coro}
\begin{proof}By \cite[Prop. 1.3]{INT13}, we can work fppf locally on $X$ and $Y$ to assume the existence of charts $\calM_X\leftarrow P\to Q \to \calM_Y$, where $P\to Q$ is of Kummer type, $P=(\calM_X/\calO_X^*)_{\bar{x}}$ for a geometric point $\bar{x}$ of $X$, and $Y \to X\times_{\Spec(\Z[P])}\Spec(\Z[Q])$ is flat. From this description, Lemma \ref{l:LocFree} implies that $P$ is free. Consequently, Lemma \ref{faithful-flatness-lemma} ensures that the induced ring map $\Z[P]\to \Z[Q]$ is flat. Finally, we apply \cite[Prop. 4.9]{Bha12} to conclude that $f$ is homologically log flat.
\end{proof}

Lemma \ref{faithful-flatness-lemma} has the following enhancement for rank 1 that we add for completeness, and we will not use in the article. 
\begin{lemm}\label{flatness-lemma}
	Let $f\colon P \hookrightarrow Q$ be an injective morphism of monoids and assume that $P$ is free of rank $1$ and $Q$ is integral. Then the induced ring homomorphism $\Z[P] \to \Z[Q]$ is flat.
\end{lemm}

\begin{proof}
	Since $P$ is free of rank $1$, we can write $\Z[P] = \Z[t]$. To show that $\Z[Q]$ is flat over $\Z[t]$, it suffices to show that for every prime number $p$, the morphism $\Z_{(p)}[t] \to \Z_{(p)}[Q]$ is flat. In turn, by \stack{0H7N}, it is enough to prove that the morphisms $\Q[t] \to \Q[Q]$ and $\F_p[t] \to \F_p[Q]$ are flat. We are reduced to prove that for a field $k$, the $k[t]$-module $k[Q]$ is torsion free. Since $Q$ is integral, we have that $k[Q]\to k[Q^\gp]$ is injective, which implies that $k[Q]$ has no $t$-torsion. After inverting $t$, we have  $$k[t]_t\subseteq k[Q]_t\subseteq k[Q^\gp]$$ and $k[Q^\gp]$ is free over $k[t]_t$, thus in particular torsion free. This yields the desired result.
\end{proof}

\begin{defi}\label{d:FrobeniusLogSmooth}Let $p$ be a prime number. 
A fine log scheme $(X,\calM_X)$ over $\Fp$ is \textit{Frobenius-smooth}\footnote{This extends the classical notion of \textit{Frobenius-smooth} schemes, as defined in \cite[Lem. 2.1.1]{DrinfeldCrystals}. By definition, Frobenius-smooth schemes are schemes over $\Fp$ which admit Zariski locally a finite absolute $p$-basis.} if étale locally on $X$ there exists a chart $P\to \calM_X$ and a morphism $X\to \mathbb{A}^n_\Fp$ for some $n\geq 0$, such that  $P$ is free and finitely generated and the induced morphism $X\to \mathbb{A}^n_{\Fp[P]}$ is relatively perfect\footnote{This is stronger than \cite[Def. 1.4]{Tsu00} and \cite[Def 1.71.1]{CV22}.}.
\end{defi}
\begin{lemm}
	If $k$ is a field of characteristic $p$ with a finite $p$-basis, $X$ is a smooth $k$-scheme, and $D$ is normal crossing divisor of $X$, then the induced log scheme is Frobenius-smooth.
\end{lemm}
\begin{proof}
	The statement is étale local. We are reduced to the case of $(\mathbb{A}^n_{k[P]},P)^a$ where $P$ is free monoid. In this case, if $k$ admits a $p$-basis of cardinality $m$, we get a relatively perfect morphism $\mathbb{A}^n_{k[P]}\to \mathbb{A}^{m+n}_{\Fp[P]}$,
	\end{proof}

\begin{prop}\label{p:FrobeniusLogSmooth}
Let $(X,\calM_X)$ be a Frobenius-smooth log scheme. Then the absolute Frobenius of $(X,\calM_X)$ is log quasi-syntomic and log flat of Kummer type. In addition, $X$ is Frobenius-smooth in the classical sense (i.e., the absolute Frobenius is syntomic). If $X$ is Noetherian, then it is also regular and excellent.
\end{prop}
\begin{proof}   Write $F \colon (X,\calM_X) \to (X,\calM_X)$ for the absolute Frobenius.
	By definition, étale locally on $X$ there exist
	a chart $P \to \calM_X$ with $P$ free and finitely generated and a morphism
	$X \to \mathbb{A}^n_{\F_p}$, such that the following diagram is cartesian with strict vertical arrows
$$\begin{tikzcd}
		(X,\calM_X) \arrow[r,"F"] \arrow[d] &  (X,\calM_X) \arrow[d] \\
		(\mathbb{A}^n_{\Fp[P]},P)^a \arrow[r,"F"]               & (\mathbb{A}^n_{\Fp[P]},P)^a .
	\end{tikzcd}$$
The lower horizontal arrow is log flat of Kummer type and homologically flat and it is syntomic on the underlying schemes. It remains to prove that the relative log cotangent complex is concentrated in $[-1,0]$. This follows from the transitivity triangle, since the log cotangent complex of $(\mathbb{A}^n_{\Fp[P]},P)^a$ over $\Fp$ is concentrated in degree $0$. Finally, last part is a consequence of Kunz's theorem \cite[Thm. 107 and Thm. 108]{Mat80}
\end{proof}
\section{Logarithmic Dieudonné theory}
\label{sec:logDieudonne}
We present here the logarithmic extension of crystalline Dieudonné theory in characteristic $p$, following the foundational work of Kato \cite{Kat23} and Inoue \cite{Ino25}. We begin by briefly recalling the classical theory of Berthelot--Breen--Messing, de Jong, and Lau (Theorem \ref{t:deJong}). The central result of this section is Theorem \ref{KatoEquivalence:t}, where we prove that the category of log $p$-divisible groups and log Dieudonné modules are equivalent on Noetherian Frobenius-smooth log schemes.

\spa

For a scheme $X$ over $\Fp$, Berthelot--Breen--Messing constructed in \cite[Déf. 3.3.6]{BBM82} a contravariant \textit{crystalline Dieudonné module functor} $$\mathbb{D}:\BT(X)^{\mathrm{op}}\to \DM(X)$$ generalising Dieudonné's original construction over perfect fields (see \cite[\S 3]{DA24} for more details).  The objects in $\DM(X)$ are triples $(\calE,\Phi_\calE,V_\calE)$ where $(\calE, \Phi_\calE)$ is a locally free crystal of finite rank and $V_\calE\colon \calE \to F^*\calE$ is a morphism such that $\Phi_\calE\circ V_\calE=p$.

\begin{theo}[{\cite{deJ95}}, \cite{Lau18}]\label{t:deJong}
Let $X$ be a Frobenius-smooth scheme over $\Fp$. The  Dieudonné module functor $$\mathbb{D}:\BT(X)^{\mathrm{op}}\to \DM(X)$$ is an exact equivalence of categories. 
\end{theo}

\subsection{}\label{s:kflVectorBundles} We switch now to the logarithmic setting. An important phenomenon in the kfl topology is that vector bundles do not satisfy descent with respect to kfl covers. The key example is provided by the Frobenius morphism. Recall that for a smooth scheme $X$ over a perfect field $k$, the classical Cartier's Frobenius descent theorem provides an equivalence between the category of vector bundles on $X^{(p)}$ and the category of vector bundles on $X$ equipped with an integrable connection with trivial $p$-curvature. In the logarithmic setting, if the log structure $\calM_{X}$ comes from a normal crossing divisor of $X$, then kfl vector bundles on $(X^{(p)},\calM_{X^{(p)}})$ correspond to kfl vector bundles on $(X,\calM_X)$ equipped with a flat log connection with trivial $p$-curvature (the Frobenius is a kfl cover by Proposition \ref{p:FrobeniusLogSmooth}). Therefore, any classical vector bundle over $(X,\calM_X)$ with a log connection with nontrivial residues corresponds to a kfl vector bundle on $(X^{(p)},\calM_{X^{(p)}})$ which is not classical.

\spa

For this reason, when defining \textit{log $p$-divisible groups} one can either work with kfl vector bundles (in \cite{Ino25} these are the weak log $p$-divisible groups) or with classical vector bundles. An additional phenomenon is that Cartier duality does not preserve representability. We use the following definition.

\begin{defi}
Let $(X,\calM_X)$ be a fine saturated log scheme over $\Fp$. We denote by $\BT^{\mathrm{log}}(X,\calM_X)$ the category of \textit{dual representable log $p$-divisible groups} over $(X,\calM_X)$ (i.e., both $\calG[p^n]$ and $\calG^\vee[p^n]$ are representable), as defined in \cite[\S 4.1]{Kat23}. Following \cite[Def. 3.13]{Ino25}, we refer to them simply as \textit{log $p$-divisible groups}.
\end{defi}
Log $p$-divisible groups have the following explicit description étale locally on $X$.

\begin{theo}[{\cite[Thm. 3.1]{Kat23}}, {\cite[\S 3.1]{WZ24}}]\label{DescriptionPDivisible-t}
Let $(X,\calM_X)$ be an fs log scheme such that $X$ is the spectrum of a Noetherian strictly henselian local ring with characteristic $p$ residue field. If $P\coloneqq\Gamma(X,\calM_X/\calO_X^*)$, then there exists a natural equivalence between the category $\BT^{\mathrm{log}}(X,\calM_X)$ and the category of pairs $(\calG,N)$ where $\calG\in \BT(X)$ and $$N\colon \calG^{\et}(1)\to \calG^{\circ}\otimes_\Z P^\gp$$ is a morphism of $p$-divisible groups\footnote{	For an étale $p$-divisible group $\calG$, we denote by $\calG(1)$ the $p$-divisible group of multiplicative type defined by $$\calG(1)_n\coloneqq \calG[p^n]\otimes_{\Z/p^n}\mu_{p^n}.$$}.  
\end{theo}

\begin{defi}
Write $(X,\calM_X)_{\mathrm{cris}}^\mathrm{log}$ for the small log crystalline site of $(X,\calM_X)$ over $(\Zp,p\Zp,\gamma^\mathrm{can})$ defined in \cite[\S 5]{Kat89}, where $\gamma^\mathrm{can}$ is the canonical PD structure on $p\Zp$. Let $\DM^\mathrm{log}(X,\calM_X)$ be the category of log $F$-crystals over $X$ which are locally free for the Zariski topology and can be endowed with a Verschiebung.
\end{defi}

Kato  constructed in \cite[\S 5]{Kat23} the log Dieudonné module functor
$$\mathbb{D}^{\mathrm{log}}\colon\BT^{\mathrm{log}}(X,\calM_X)^\mathrm{op}\to \DM^\mathrm{log}(X,\calM_X),$$

 under a technical assumption on the existence of PD structures for kfl covers [\textit{ibid.} \S 5.2.1]. If $(X,\calM_X)$ is locally free, we get rid of this extra assumption thanks to the following lemma. 
 
 \begin{lemm}\label{l:Assumption}
Let $(X,\calM_X)$ be a log scheme with a locally free log structure. For every log (strict) PD thickening $((U,\calM_U)\hookrightarrow (T,\calM_T),\gamma)\in (X,\calM_X)_{\mathrm{cris}}^\mathrm{log}$ and every log flat morphism $(T',\calM_{T'})\to (T,\calM_T)$ of Kummer type, there exists a unique PD structure on the ideal associated to $U\times_T T'\hookrightarrow T'$ which is compatible with $\gamma$.
 \end{lemm}
 \begin{proof}
Thanks to Lemma \ref{l:Thick}, the log structure $\calM_T$ is locally free. We deduce from Corollary \ref{c:KummerClassicallyFlat} that $T'\to T$ is classically flat. The result then follows from \stack{07H1}.
 \end{proof}

We are now ready to state the fundamental theorem of logarithmic Dieudonné theory, which provides a crystalline classification of log $p$-divisible groups.




\begin{theo}[de Jong, Lau, Kato, Inoue]\label{KatoEquivalence:t} If $(X,\calM_X)$ is a Noetherian Frobenius-smooth log scheme, the log Dieudonné module functor
	 $$\mathbb{D}^{\mathrm{log}}\colon\BT^{\mathrm{log}}(X,\calM_X)^\mathrm{op}\to \DM^\mathrm{log}(X,\calM_X)$$ is an exact equivalence of categories.
\end{theo}
\begin{proof}In the case when $(X,\calM_X)$ is log smooth over a perfect field of odd characteristic, the proof of the equivalence is sketched in \cite[Thm. 6.3]{Kat23}. The additional condition that $p\neq 2$ is needed there since he uses the results in \cite{BK91}. One can replace [\textit{ibid.}] with Theorem \ref{t:deJong} to include the case $p=2$. 
	
\spa

Inoue proves in \cite[Prop. 5.7.3]{Ino25} that $\mathbb{D}^{\mathrm{log}}$ is an equivalence when $X$ is quasi-syntomic and étale locally on $X$ there exists a free chart $P\to \calM_{X}$ such that $\tilde{X}\coloneqq X\otimes_{\Z[P]} \Z[\tfrac{1}{\infty}P]$ has a perfect quasi-syntomic cover, where $\tfrac{1}{\infty}P\coloneqq \varinjlim_n \tfrac{1}{n}P$. Let us explain why the assumption is satisfied in our situation. We take as $P$ any free chart of $(X,\calM_X)$ and we consider $$X_{\mathrm{perf}}\coloneqq\varprojlim \left(\dots \xrightarrow{F}X\xrightarrow{F} X\right).$$ Again, by Proposition \ref{p:FrobeniusLogSmooth}, the scheme $X_{\mathrm{perf}}$ is quasi-syntomic over $X$. In addition, the scheme $$\tilde{X}_\mathrm{perf}\coloneqq X_{\mathrm{perf}}\otimes_{\Z\left[\tfrac{1}{p^\infty}P\right]} \Z[\tfrac{1}{\infty}P]$$ is perfect, since it is semi-perfect and pro-étale over $X_{\mathrm{perf}}$. The quasi-syntomic cover $\tilde{X}_\mathrm{perf}\to \tilde{X}$ satisfies the desired property.\end{proof}

Over henselian discrete valuation rings, dual representable log $p$-divisible groups correspond to semi-stable $p$-divisible groups, as introduced in \cite{deJ98}, \cite{deJ99}.

\begin{defi}
	Let $R$ be a henselian discrete valuation ring with fraction field $\Omega$. We say that the generic fibre $\calG_\eta\in \BT(\Omega)$ has \textit{good reduction} if it extends to a \bt group over $\Spec(R)$. In this case, we say that $\calG$ is the\textit{ model} of $\calG_\eta$. We also say that $\calG_\eta\in \BT(\Omega)$ has \textit{semi-stable reduction} if there exists a filtration $$ \calG_\eta^\mu\subseteq \calG_\eta^f\subseteq \calG_\eta$$ such that $\calG_\eta^f$ and $\calG_\eta/\calG_\eta^\mu$ have models $\calG_1$ and $\calG_2$, the kernel of the induced map $\calG_1\to \calG_2$ is a $p$-divisible group of multiplicative type, and the cokernel is an étale $p$-divisible group.  We denote by $\BT^{\mathrm{st}}(\Omega)\subseteq \BT(\Omega) $ the category they form.  We also denote by $\BT^{\mathrm{pst}}(\Omega)$ the full subcategory of $\BT(\Omega)$ of \textit{potentially semi-stable} $p$-divisible groups over $\Spec(\Omega)$, namely those $p$-divisible groups which acquire semi-stable reduction after an fppf cover. If $X$ is a smooth curve over a field, we say that a $p$-divisible group over $X$ is semi-stable or potentially semi-stable if it is so when restricted to the fraction fields of the completions at all the points at infinity.
\end{defi}

\begin{theo}[{\cite[Thm. 5.19]{BWZ23}}]\label{BWZ:t}
 Let $R$ be an henselian discrete valuation ring with residue field of characteristic $p$. The restriction to the generic fibre defines an equivalence $$
	\BT^\mathrm{log}(R,M_R)\iso \BT^{\mathrm{st}}(\Omega),$$ where $M_R$ is the divisorial log structure associated to the maximal ideal.
\end{theo}

\section{The overconvergent Dieudonné equivalence}
\label{sec:overconvequiv}In this section, we prove the key equivalence between potentially semi-stable $p$-divisible groups and overconvergent $F$-isocrystals with bounded slopes (Theorem \ref{DM:t}), generalising earlier results of \cite{Tri08} and \cite{Pal22}. We begin with Proposition \ref{DMIsoc:p}, which identifies over curves the isogeny category of Dieudonné modules with coherent $F$-isocrystals having slopes in $[0,1]$. Building on the log Dieudonné theory of the previous section and Trihan's work \cite{Tri08}, we define the overconvergent Dieudonné module functor $\mathbb{D}^\ag$ and prove its essential properties.
\begin{prop}\label{DMIsoc:p}If $(X,\calM_X)$ is a Noetherian Frobenius-smooth log scheme of Krull dimension $\leq 1$, then there is a natural equivalence $$\DM^\mathrm{log}(X,\calM_X)_\Qp\iso \Fisoc^\mathrm{log}(X,\calM_X)^{[0,1]},$$ where $\Fisoc^\mathrm{log}(X,\calM_X)^{[0,1]}\subseteq \Fisoc^\mathrm{log}(X,\calM_X)$ is the full subcategory of coherent log $F$-isocrystals with slopes in $[0,1]$.
\end{prop}

\begin{proof}
	We have to prove that every coherent log $F$-isocrystal $(\calE,\Phi_\calE)$ in $\Fisoc^\mathrm{log}(X,\calM_X)^{[0,1]}$ is isogenous to a log Dieudonné module.  
	Up to replacing $\calE$ with its double dual, we may assume $\calE$ reflexive.  
	Let $\calE_V$ be the log subcrystal of $\calE[\tfrac{1}{p}]$ generated by $V^m_\calE(\calE)$ for all $m\geq 0$, where  
	$V_\calE \colon F^*\calE\to \calE[\tfrac{1}{p}]$ is defined by $V_{\calE}\coloneqq p\Phi_{\calE}^{-1}$.  
	Note that $\calE_V$ is a log $F$-crystal because $\Phi_\calE \circ V_\calE^{m+1}=pV_\calE^{m}$.  
	We want to prove that $\calE_V$ is coherent.
	
	\spa
	
	 By Proposition \ref{p:FrobeniusLogSmooth}, we may assume $X=\Spec(A_0)$ where $A_0$ is a regular domain which admits a flat $p$-complete lift $A$ over $\Zp$. Note that $A$ is regular as well of Krull dimension at most $2$. By \stack{0B3N}, this implies that $E$ is projective. Write $(E,\Phi_E)$ for the integrable log connection with Frobenius over $A$ associated to $(\calE,\Phi_\calE)$. We want to prove that there exists $n>0$ such that 
	\begin{equation}\label{Condition:eq}
		p^{m+n}E\subseteq \Phi_E^m(E)
	\end{equation} for all $m\geq 0$. This would imply that $V_E^m(E)\subseteq \tfrac{1}{p^n}E$ for every $m$, and this, in turn, would imply that $E_V$ is coherent since $\tfrac{1}{p^n}E$ is so.
	
	\spa
	Choose an embedding of $A_0$ in an algebraically closed field $\Omega$. This induces a $p$-saturated inclusion $A\hookrightarrow W(\Omega)$ (i.e., $pW(\Omega)\cap A=pA$). Since $E$ is projective, we deduce that $E\hookrightarrow  E_\Omega\coloneqq E\otimes_A W(\Omega)$ is also $p$-saturated and injective. Combining this with the fact that $\Phi_E[\tfrac{1}{p}]$ is an isomorphism, we deduce that for every $m\geq 1$, we have $$\Phi^{-m}_{E_\Omega}(E)\subseteq (F^m)^*(E[\tfrac{1}{p}]\cap E_\Omega)=(F^m)^*E.$$ This shows that in order to prove \eqref{Condition:eq}, it is enough to prove that $$p^{m+n}E_\Omega\subseteq \Phi_{E_\Omega}^m(E_\Omega).$$ We can now invoke Dieudonné--Manin classification and verify this condition on elementary $F$-crystals over $\Omega$. For the one of slope $r/s$ with $(r,s)=1$, it is enough to take $n=s-1$.
\end{proof}

\begin{cons}Let $X$ be a smooth curve over a perfect field $k$. Trihan proved in \cite{Tri08} that the restriction of $\mathbb{D}_\Qp$ to $\BT^{\mathrm{st}}(X)_\Qp^{\mathrm{op}}$ factors through $$\iota\colon \Foi(X)\hookrightarrow \Fisoc(X).$$ We then obtain a $\Qp$-linear functor
	$$\mathbb{D}^\ag:\BT^{\mathrm{st}}(X)_\Qp^{\mathrm{op}}\to \Foi(X).$$ By fpqc descent, we obtain a fully faithful functor $$\mathbb{D}^\ag:\BT^{\mathrm{pst}}(X)_\Qp^{\mathrm{op}}\to \Foi(X)$$ for $X$ a smooth curve over any characteristic $p$ field (not assumed to be perfect), which we call the \textit{overconvergent Dieudonné module functor}. Similarly, we have a fully faithful functor  	$$\mathbb{D}^\ag:\BT^{\mathrm{pst}}(k((t)))_\Qp^{\mathrm{op}}\to\Foi(k((t))),$$ where $\Foi(k((t)))$ is as in \cite[Rmk. 2.11]{Ked16}.
\end{cons}

\begin{theo}\label{DM:t}
	Let $k$ be a characteristic $p$ field with a finite $p$-basis and let $X$ be a smooth curve over $k$. The contravariant overconvergent Dieudonné module functor induces an equivalence
	$$\mathbb{D}^\ag:\BT^{\mathrm{pst}}(X)_\Qp^{\mathrm{op}}\iso \Foi(X)^{[0,1]}$$ of $\Qp$-linear abelian categories. Similarly, we have an equivalence $$\mathbb{D}^\ag:\BT^{\mathrm{pst}}(k((t)))_\Qp^{\mathrm{op}}\iso\Foi(k((t)))^{[0,1]}$$ of $\Qp$-linear abelian categories.
\end{theo}
\begin{proof}We prove the first part; the second part is proven analogously. The functor $\mathbb{D}^\ag$ is fully faithful thanks to Theorem \ref{t:deJong}. It remains to prove the essential surjectivity.  Consider $(\calE^\ag,\Phi_{\calE^\ag})\in \Foi(X)^{[0,1]}$, we need to find a potentially semi-stable $p$-divisible group $\calG$ over $X$ (up to isogeny) such that $\mathbb{D}^\ag(\calG) \simeq (\calE^\ag, \Phi_{\calE^\ag})$. Let $Y$ be a smooth compactification of $X$ and $\calM_Y$ the log structure associated to the divisor at infinity. By \cite{KedCurve}, there exists an fppf cover of $X$ such that the base change of  $(\calE^\ag,\Phi_{\calE^\ag})$ is in the essential image of $$\Fisoc^{\mathrm{log}}(Y,\calM_Y)^{[0,1]}\to \Foi(X)^{[0,1]}.$$ By Proposition \ref{DMIsoc:p}, we deduce that it is also in the essential image of $$\DM^\mathrm{log}(Y,\calM_Y)\to \Foi(X)^{[0,1]}.$$ By Theorem \ref{KatoEquivalence:t}, we deduce that there exists a log $p$-divisible group $\calG^\mathrm{log}$ corresponding to $(\calE^\ag, \Phi_{\calE^\ag})$.  Thanks to Theorem \ref{BWZ:t}, the restriction of $\calG^\mathrm{log}$ to $X$ is semi-stable. This ends the proof.\end{proof}

\section{Main theorem}
\label{sec:main}
We now combine the results of the previous sections to prove our main theorem on the overconvergence of isoclinic quotients of minimal slope. The proof hinges on Tate's local extension property for étale subgroups of $p$-divisible groups (Proposition \ref{good-red-et:p}), which we review in detail. Corollary \ref{local-et-sub:c} extends this property to the semi-stable setting. Using the overconvergent Dieudonné equivalence of Theorem \ref{DM:t} and Tsuzuki's characterization of unit-root overconvergent $F$-isocrystals \cite{TsuFLM}, we prove Theorem \ref{a-UR01:t}.


\begin{prop}[{\cite[Prop. 12]{Tat67}}]\label{good-red-et:p}Let $R$ be a discrete valuation ring. For a $p$-divisible group $\calG_R$ over $R$, every étale \bt subgroup $\calH_\eta\subseteq\calG_\eta$ of the generic fibre has good reduction.
\end{prop}
\begin{proof}
    While Theorem 4 in \cite{Tat67} requires characteristic $0$ and its proof does not directly translate to characteristic $p$ (where, to our knowledge, only proofs using overconvergent techniques exist \cite{deJ98}), the characteristic assumption is unnecessary for this specific proposition. The proof is reproduced here for completeness.

\spa

For each $n$, let $\mathcal{H}_n$ be the scheme-theoretic closure of $\mathcal{H}_\eta[p^n]$ in $\mathcal{G}[p^n]$. The scheme $\mathcal{H}_n$ defines a finite flat subgroup scheme of $\mathcal{G}[p^n]$ over $R$ such that $\mathcal{H}_{n,\eta} = \mathcal{H}_\eta[p^n]$. In addition, the inclusions $\mathcal{G}[p^n] \hookrightarrow \mathcal{G}[p^{n+1}]$ induce closed immersions $\mathcal{H}_n \hookrightarrow \mathcal{H}_{n+1}$. We obtain an inductive system $(\mathcal{H}_n)_{n\in\mathbb{N}}$ which may not yet be $p$-divisible. Define the finite flat group schemes $\mathcal{D}_n$ over $R$ by
 $\mathcal{H}_{n+1} / \mathcal{H}_n.$
On the generic fiber, we have
$$\mathcal{D}_{n,\eta} = \mathcal{H}_\eta[p^{n+1}] / \mathcal{H}_\eta[p^n] \iso \mathcal{H}_\eta[p],$$
which is \'{e}tale. Since $\mathcal{D}_n$ is flat and the generic fibre is killed by $p$, we also have $p\mathcal{D}_n=0$. If $A_n$ is the $R$-algebra associated to $\mathcal{D}_n$, then it is a finite flat $R$-algebra such that $A_n \otimes_R \Omega$ is isomorphic to the $\Omega$-algebra of $\mathcal{H}_\eta[p]$, denoted by $A_\eta$. We deduce that there is a natural inclusion of $R$-orders $$ A_0 \subseteq A_1 \subseteq A_2 \subseteq \cdots $$ inside $A_\eta$. Since $A_\eta$ is finite-dimensional and étale over $\Omega$, the integral closure of $R$ inside $A_\eta$ is a finite $R$-algebra \stack{032L}. Therefore, the sequence stabilizes: there exists $n_0 \in \mathbb{N}$ such that $A_n = A_{n_0}$ for all $n \geq n_0.$ This implies that the morphism $\mathcal{D}_{n+n_0}\to \mathcal{D}_{n_0}$ induced by the multiplication by $p^n$ is an isomorphism for every $n\geq 0$.

\spa

We define a new system $(\tilde{\mathcal{H}}_n)_{n\in \mathbb{N}}$ by
$\tilde{\mathcal{H}}_n \coloneqq \mathcal{H}_{n_0+n} / \mathcal{H}_{n_0}.$
Each $\tilde{\mathcal{H}}_n$ is a finite flat group scheme over $R$ and the multiplication by $p^{n_0}$ induces an isomorphism
$$\tilde{\mathcal{H}}_{n,\eta} = \mathcal{H}_\eta[p^{n_0 + n}] / \mathcal{H}_\eta[p^{n_0}] \simeq \mathcal{H}_\eta[p^n].$$ We want to show that $(\tilde{\mathcal{H}}_n)_{n\in \mathbb{N}}$ is $p$-divisible; it is enough to show that for each $n \geq 0$, the sequence
$$0 \to \tilde{\mathcal{H}}_n \to \tilde{\mathcal{H}}_{n+1} \xrightarrow{p^n} \tilde{\mathcal{H}}_{1} \to 0$$
is exact. The multiplication-by-$p^n$ map
$$p^n : \tilde{\mathcal{H}}_{n+1} = \mathcal{H}_{n_0+n+1} / \mathcal{H}_{n_0} \to \mathcal{H}_{n_0+1} / \mathcal{H}_{n_0} = \tilde{\mathcal{H}}_{1}$$
factors as $$\tilde{\mathcal{H}}_{n+1} \xrightarrow{\alpha} \mathcal{H}_{n_0+n+1} / \mathcal{H}_{n_0+n}=\mathcal{D}_{n_0+n} \xrightarrow{\beta} \mathcal{D}_{n_0}=\tilde{\mathcal{H}}_{1},$$
where $\alpha$ is the canonical projection and $\beta$ is induced by the multiplication by $p^n$. By construction, $\beta$ is an isomorphism. Hence, the kernel of $p^n = \beta \circ \alpha$ is the same as the kernel of $\alpha$, which is
$$\mathcal{H}_{n_0+n} / \mathcal{H}_{n_0} = \tilde{\mathcal{H}}_n.$$
Thus, $(\tilde{\mathcal{H}}_n)_{n\in \mathbb{N}}$ is an integral model of $\mathcal{H}_\eta$, as we wanted.
\end{proof}

\begin{coro}\label{local-et-sub:c}Let $R$ be a henselian discrete valuation ring of characteristic $p$ and let $\calG_\eta$ be a \bt group over $\Omega$ with semi-stable reduction. Every étale \bt subgroup $\calH_\eta\subseteq \calG_\eta$ has good reduction.
\end{coro}

\begin{proof}
Let $\calG_\eta^\mu\subseteq \calG_\eta^f\subseteq \calG_\eta$ be the filtration associated to the semi-stable model. The group $\calG_\eta^\mu$ is contained in the generic fibre of $\mathrm{Ker}(\calG_1\to \calG_2)$, thus it is of multiplicative type. Since $R$ is of characteristic $p$, the natural map $\calH_\eta\to \calG_\eta/\calG_\eta^\mu$ is then injective. By Proposition \ref{good-red-et:p}, since $\calG_\eta/\calG_\eta^\mu$ has good reduction, $\calH_\eta$ has good reduction as well. This proves the desired result.
\end{proof}



We are now ready to prove the main theorem. Let $k$ be a perfect field of positive characteristic $p$ and let $X$ be a smooth $k$-scheme.

\begin{theo}\label{a-UR01:t}

	Let $\calE$ be a $\dag$-extendable $F$-isocrystal over $X$ with slopes in the interval $[s,s+1]$ for some $s\in \Q$. Every isoclinic quotient $(\calE,\Phi_\calE)\twoheadrightarrow (\calF,\Phi_\calF)$ of slope $s$ is $\dag$-extendable. In addition, if $\calE$ extends to a log $F$-isocrystal over a smooth compactification $X\subseteq Y$ with $Y\setminus X$ a normal crossing divisor, then $(\calF,\Phi_\calF)$ extends to $Y$. 
\end{theo}

\begin{proof}By descent, we may assume the base field algebraically closed and $X$ connected. After taking powers of $F$ and twisting by a rank 1 constant $F$-isocrystal, we may further assume $s=0$. Since $(\calF,\Phi_\calF)$ is unit-root, thanks to \cite[Thm. 2.1]{Crew87}, it corresponds to a finite dimensional $L$-linear continuous representation $\rho\colon \pi_1^\et(X)\to \GL(V)$, where $L$ is a finite extension of $\Qp$. By \cite[Thm. 2.3.7]{KedSwanII} (which is an enhancement of \cite{TsuFLM}), it is enough to prove that for every field $\Omega=\Omega_0((t))$ with $\Omega_0$ algebraically closed and every morphism $f\colon \Spec(\Omega)\to X$, the image $$\rho (f_*(\Gal(\bar{\Omega}/\Omega)))\subseteq \GL(V)$$ is finite. 
	
	\spa
	
	The inverse image of $\calE$ to $\Spec(\Omega)$, denoted by $\calE_\Omega$, is in the essential image of the functor $$\iota\colon \Foi(\Omega)\to \Fisoc(\Omega)$$ and the slopes are in the interval $[0,1]$. By Theorem \ref{DM:t}, there exists a potentially semi-stable \bt group $\calG$ over $\Omega$ such that $\iota(\mathbb{D}^\ag(\calG))\simeq (\calE_\Omega,\Phi_{\calE_\Omega})$. Moreover, by Theorem \ref{t:deJong}, the quotient $\calE_\Omega\twoheadrightarrow \calF_\Omega$ corresponds via $\mathbb{D}_\Qp$ to an étale \bt subgroup $\calH\subseteq \calG$ over $\Omega$, which is unique up to isogeny. By Corollary \ref{local-et-sub:c}, the \bt group $\calH$ has potential good reduction. This implies that $(\calF_\Omega,\Phi_\Omega)$ extends to $\Omega_0[[t]]$ after passing to a finite extension of $\Omega$. We deduce that $\rho (f_*(\Gal(\bar{\Omega}/\Omega)))$ is finite. This proves the first part. The second part is proved analogously combining Corollary \ref{local-et-sub:c} and Zariski--Nagata purity theorem, \cite[Thm. X.3.1]{SGA1}.
\end{proof}
We end the article stating a variant of Theorem \ref{a-UR01:t} with the language of log $p$-divisible groups.
\begin{theo}\label{et-sub-psemi:t}
	Let $(Y,D)$ be a smooth compactification of $X$, where $D$ is a normal crossing divisor, and let $\calG$ be a log $p$-divisible group over $(Y,D)$.  Then every étale $p$-divisible subgroup $\calH\subseteq\calG_X$ extends to a $p$-divisible group over $Y$.
\end{theo}

\begin{proof}
This result follows from Theorem \ref{a-UR01:t} thanks to Theorem \ref{KatoEquivalence:t}. More directly, one can rather mimic the proof of Theorem \ref{a-UR01:t} and avoid Dieudonné theory.
\end{proof}
\begin{rema}In \cite[Ex. 2.2]{Kos17} and \cite[\S 3]{Hel22}, the authors construct nonconstant étale $p$-divisible subgroups of abelian schemes. These are obtained by looking at the reduction modulo $p$ of certain Shimura varieties of PEL type. Note that, by Theorem \ref{AbelianIntro:t}, the existence of nontrivial étale subgroups implies that the (rationalised) endomorphism algebra of the abelian variety at the generic point has at least one factor of Albert-type IV.
\end{rema}

	\bibliographystyle{ams-alpha}

\end{document}